\documentclass{amsart}
\usepackage{amssymb}

\usepackage{amsfonts}

\newtheorem{theorem}{Theorem}
\theoremstyle{plain}

\newtheorem{corollary}{Corollary}

\newtheorem{definition}{Definition}

\numberwithin{equation}{section}
\input{tcilatex}

\begin{document}
\title[Involute Curves]{SOME CHARACTERIZATIONS FOR THE INVOLUTE CURVES IN
DUAL SPACE}
\author{Suleyman SENYURT}
\address{Departman of Mathematics Faculty of Art and Science Ordu University}
\email{senyurtsuleyman@hotmail.com}
\thanks{Corresponding Author}
\author{Mustafa BILICI}
\curraddr{Departman of Mathematics Faculty of Education Ondokuz Mayis
University}
\email{mbilici@omu.edu.tr}
\thanks{Thanks }
\author{Mustafa CALISKAN}
\address{Departman of Mathematics Faculty of Art and Science Ondokuz Mayis
University}
\subjclass[2000]{ 53A04; 45F10}
\keywords{Dual cuve, Involute, Evolute, Dual space}

\begin{abstract}
In this paper, we investigate some characterizations of involute -- evolute
curves in dual space. Then the relationships between dual frenet frame and
darboux vectors of these curves are found.
\end{abstract}

\maketitle

\section{Introduction}

Foundation notion with respect to involute of a curve are given in
3-dimensional Euclidean space $IR^{3}$, \cite{3.},\cite{8.}. The
relationships between frenet frames of involute - evolute curves and some
characterizations related to these curves in 3-dimensional Euclidean space $%
IR^{3}~$and Minkowski space $IR_{1}^{3}~$are described by Cal\i skan and
Bilici, \cite{1.},\cite{2.},\cite{3.}.

In this study some new characterizations with respect to involute -- evolute
curves $ID^{3}~$in are given.

\section{Preliminaries}

The set $~ID=\left\{ \overset{\Lambda }{a}=a+\varepsilon a^{\ast }\mid
a,a^{\ast }\epsilon IR,~\epsilon ^{2}=0\right\} ~$is called dual numbers set
by W.K. Clifford (1849-79) as a tool for his geometrical investigations.

Product and addition operations on this set are described respectively,

\begin{eqnarray*}
\left( a+\varepsilon a^{\ast }\right) +\left( b+\varepsilon b^{\ast }\right)
&=&\left( a+b\right) +\varepsilon \left( a^{\ast }+b^{\ast }\right) \\
\left( a+\varepsilon a^{\ast }\right) .\left( b+\varepsilon b^{\ast }\right)
&=&ab+\varepsilon \left( ab^{\ast }+a^{\ast }b\right)
\end{eqnarray*}

Algebric construction $\left( ID,+,.\right) $is unit and commutative ring.

Addition and scalar product on $ID^{3}=\left\{ \overset{\overset{\rightarrow 
}{\Lambda }}{a}\mid \overset{\overset{\rightarrow }{\Lambda }}{a}=\overset{%
\rightarrow }{a}+\varepsilon \overset{\rightarrow }{a^{\ast }},~\overset{%
\rightarrow }{a},\overset{\rightarrow }{a^{\ast }}\epsilon IR^{3}\right\} ~$%
are described.

\begin{eqnarray*}
\oplus &:&ID^{3}\times ID^{3}\rightarrow ID^{3}~,~\overset{\overset{%
\rightarrow }{\Lambda }}{a}\oplus \overset{\overset{\rightarrow }{\Lambda }}{%
b}=\left( \overset{\rightarrow }{a}+\overset{\rightarrow }{b}\right)
+\epsilon \left( \overset{\rightarrow }{a^{\ast }}+\overset{\rightarrow }{%
b^{\ast }}\right) \\
\odot &:&ID\times ID^{3}\rightarrow ID^{3}~,~~~~~\overset{\Lambda }{\lambda }%
\odot \overset{\overset{\rightarrow }{\Lambda }}{a}=\lambda \overset{%
\rightarrow }{a}+\epsilon \left( \lambda \overset{\rightarrow }{a^{\ast }}%
+\lambda ^{\ast }\overset{\rightarrow }{a}\right)
\end{eqnarray*}

Algebric construction $\left( ID^{3},\oplus ,ID,+,.,\odot \right) $ is a
modul. This modul is called $ID-Modul$.

The inner product and vectorel product of dual vectors $\overset{\overset{%
\rightarrow }{\Lambda }}{a},\overset{\overset{\rightarrow }{\Lambda }}{b}%
\epsilon ID^{3}~$are defined by respectively,

\begin{eqnarray*}
\left\langle ~,\right\rangle &:&ID^{3}\times ID^{3}\rightarrow
ID~,~\left\langle \overset{\overset{\rightarrow }{\Lambda }}{a},\overset{%
\overset{\rightarrow }{\Lambda }}{b}\right\rangle =\left\langle \overset{%
\rightarrow }{a},\overset{\rightarrow }{b}\right\rangle +\varepsilon \left(
\left\langle \overset{\rightarrow }{a},\overset{\rightarrow }{b^{\ast }}%
\right\rangle +\left\langle \overset{\rightarrow }{a^{\ast }},\overset{%
\rightarrow }{b}\right\rangle \right) \\
\wedge &:&ID^{3}\times ID^{3}\rightarrow ID^{3}~,~\overset{\overset{%
\rightarrow }{\Lambda }}{a}\wedge \overset{\overset{\rightarrow }{\Lambda }}{%
b}=\left( \overset{\rightarrow }{a}\wedge \overset{\rightarrow }{b}\right)
+\varepsilon \left( \overset{\rightarrow }{a}\wedge \overset{\rightarrow }{%
b^{\ast }}+\overset{\rightarrow }{a^{\ast }}\wedge \overset{\rightarrow }{b}%
\right)
\end{eqnarray*}

For $~\overset{\overset{\rightarrow }{\Lambda }}{a}\neq 0$, the norm $%
\left\Vert \overset{\overset{\rightarrow }{\Lambda }}{a}\right\Vert $of $~%
\overset{\overset{\rightarrow }{\Lambda }}{a}=\overset{\rightarrow }{a}%
+\varepsilon \overset{\rightarrow }{a^{\ast }}~$is defined by

\begin{equation*}
\left\Vert \overset{\overset{\rightarrow }{\Lambda }}{a}\right\Vert =\sqrt{%
\left\langle \overset{\overset{\rightarrow }{\Lambda }}{a},\overset{\overset{%
\rightarrow }{\Lambda }}{a}\right\rangle }=\left\Vert \overset{\rightarrow }{%
a}\right\Vert +\varepsilon \frac{\left\langle \overset{\rightarrow }{a},%
\overset{\rightarrow }{a^{\ast }}\right\rangle }{\left\Vert \overset{%
\rightarrow }{a}\right\Vert }~~,~\left\Vert \overset{\rightarrow }{a}%
\right\Vert \neq 0.~
\end{equation*}%
\qquad \qquad

The angle between unit dual vectors $\overset{\overset{\rightarrow }{\Lambda 
}}{a\text{ }}$and $\overset{\overset{\rightarrow }{\Lambda }}{b}~\Phi
=\varphi +\varepsilon \varphi ^{\ast }~$is called dual angle and this angle
is denoted by

\begin{equation*}
\left\langle \overset{\overset{\rightarrow }{\Lambda }}{a},\overset{\overset{%
\rightarrow }{\Lambda }}{b}\right\rangle =\text{cos}\left( \Phi \right) =%
\text{cos}\left( \varphi \right) -\varepsilon \varphi ^{\ast }\text{sin}%
\left( \varphi \right)
\end{equation*}%
\qquad \qquad

Let

\begin{eqnarray*}
\overset{\Lambda }{\alpha } &:&I\subset IR\rightarrow ID^{3} \\
\text{ }s &\rightarrow &\overset{\Lambda }{\alpha }\left( s\right) =\alpha
\left( s\right) +\varepsilon \alpha ^{\ast }\left( s\right)
\end{eqnarray*}

be differential unit speed dual curve in dual space $ID^{3}$. Denote by $%
\left\{ T,N,B\right\} ~$the moving dual frenet frame along the dual space
curve $\overset{\Lambda }{\alpha }\left( s\right) ~$in the dual space $ID^{3}
$. Then $T,N~$and $B~$are the dual tangent, the dual principal normal and
the dual binormal vector fields, respectively. The function $\kappa \left(
s\right) =k_{1}+\varepsilon k_{1}^{\ast }~$and $\tau \left( s\right)
=k_{2}+\varepsilon k_{2}^{\ast }~$are called dual curvature and dual torsion
of $\overset{\Lambda }{\alpha }$, respectively. Then for the dual curve $%
\overset{\Lambda }{\alpha }~$the frenet formulae are given by,

\begin{equation}
\left\{ 
\begin{array}{c}
T^{^{\prime }}\left( s\right) =\kappa \left( s\right) N\left( s\right) \\ 
N^{^{\prime }}\left( s\right) =-\kappa \left( s\right) T\left( s\right)
+\tau \left( s\right) B\left( s\right) \\ 
B^{^{\prime }}\left( s\right) =-\tau \left( s\right) N\left( s\right)%
\end{array}%
\right.  \label{2.1}
\end{equation}%
$~\ \ \ \ \ \ \ \ \ \ \ \ \ \ \ \ \ \ \ \ \ \ \ \ \ \ \ \ \ \ \ \ \ \ \ \ \
\ \ \ \ \ \ \ \ \ \ \ \ \ \ \ \ \ \ \ \ \ \ \ \ \ \ \ \ \ \ \ \ \ \ \ \ \ \
\ \ \ \ \ \ \ \ \ \ \ \ \ \ \ \ \ \ $

The formulae (\ref{2.1}) are called the frenet formulae of dual curve in 
\cite{9.}. In this palace curvature and torsion are calculated by,

\begin{equation}
\kappa \left( s\right) =\sqrt{\left\langle T^{^{\prime }},T^{^{\prime
}}\right\rangle }~,~\tau \left( s\right) =\frac{\text{det}\left(
T,T^{^{\prime }},T^{^{\prime \prime }}\right) }{\left\langle T^{^{\prime
}},T^{^{\prime }}\right\rangle }  \label{2.2}
\end{equation}

If $\alpha ~$is not unit speed curve, then curvature and torsion are
calculated by,

\begin{equation}
\kappa \left( s\right) =\frac{\left\Vert \alpha ^{^{\prime }}\left( s\right)
\wedge \alpha ^{^{\prime \prime }}\left( s\right) \right\Vert }{\left\Vert
\alpha ^{^{\prime }}\left( s\right) \right\Vert ^{3}}~~,~\tau \left(
s\right) =\frac{\text{det}\left( \alpha ^{^{\prime }}\left( s\right) ,\alpha
^{^{\prime \prime }}\left( s\right) ,\alpha ^{^{\prime \prime \prime
}}\left( s\right) \right) }{\left\Vert \alpha ^{^{\prime }}\left( s\right)
\wedge \alpha ^{^{\prime \prime }}\left( s\right) \right\Vert ^{2}}
\label{2.3}
\end{equation}

If formulae (\ref{2.1}) is separated into the real and dual part, we can
obtain

\begin{equation}
\left\{ 
\begin{array}{c}
t^{^{\prime }}\left( s\right) =k_{1}n \\ 
n^{^{\prime }}\left( s\right) =-k_{1}t+k_{2}b \\ 
b^{^{\prime }}\left( s\right) =-k_{2}n%
\end{array}%
\right.  \label{2.4}
\end{equation}

\begin{equation}
\left\{ 
\begin{array}{c}
t^{\ast ^{\prime }}\left( s\right) =k_{1}n^{\ast }+k_{1}^{\ast }n \\ 
n^{\ast ^{\prime }}\left( s\right) =-k_{1}t^{\ast }-k_{1}^{\ast
}t+k_{2}b^{\ast }+k_{2}^{\ast }b \\ 
b^{\ast ^{\prime }}\left( s\right) =-k_{2}n^{\ast }-k_{2}^{\ast }n%
\end{array}%
\right.  \label{2.5}
\end{equation}

\section{Some Characterizations Involute of Dual Curves}

\begin{definition}
Let $\alpha :I\rightarrow ID^{3}~$and $\beta :I\rightarrow ID^{3}~$be dual
unit speed curves. If the tangent lines of the dual curve $\alpha $ is
orthogonal to the tangent lines of the dual curve $\beta $ , the dual curve $%
\beta $is called involute of the dual curve $\alpha $ or the dual curve $%
\alpha ~$is called evolute of the dual curve $\beta $. According to this
definition, if the tangent of the dual curve $\alpha $ is denoted by $T~$and
the tangent of the dual curve $\beta $ is denoted by $\overset{-}{T}$, we
can write
\end{definition}

\begin{equation}
\left\langle T,\overset{-}{T}\right\rangle =0  \label{3.1}
\end{equation}

\begin{theorem}
Let $\alpha $ and $\beta ~$be dual curves. If the dual curve~$\beta $
involute of the dual curve $\alpha $, we can write
\end{theorem}

\begin{equation*}
\beta \left( s\right) =\alpha \left( s\right) +\left[ \left( c_{1}-s\right)
+\varepsilon c_{2}\right] T\left( s\right) ~\,,~~c_{1},c_{2}\epsilon IR.
\end{equation*}

\begin{proof}
Then by the definition we can assume that

\begin{equation}
\beta \left( s\right) =\alpha \left( s\right) +\lambda T\left( s\right)
~\,~,~~\lambda \left( s\right) =\mu \left( s\right) +\varepsilon \mu ^{\ast
}\left( s\right)  \label{3.2}
\end{equation}

For some function $\lambda \left( s\right) $. By taking derivate of the
equation (\ref{3.2}) with respect to $s$ and applying the frenet formulae (%
\ref{2.1}) we have

\begin{equation*}
\frac{d\beta }{ds}=\left( 1+\frac{d\lambda }{ds}\right) T+\lambda \kappa N
\end{equation*}

where $s$ and $s^{\ast }~$are arc parameter of the dual curves $\alpha ~$and 
$\beta $, respectively,

\begin{equation}
\overset{-}{T}\frac{ds^{\ast }}{ds}=\left( 1+\frac{d\lambda }{ds}\right)
T+\lambda \kappa N  \label{3.3}
\end{equation}

Taking the inner product of (\ref{3.3}) with $T$~we have

\begin{equation}
\frac{ds^{\ast }}{ds}\left\langle T,\overset{-}{T}\right\rangle =\left( 1+%
\frac{d\lambda }{ds}\right) \left\langle T,T\right\rangle +\lambda
\left\langle T,N\right\rangle  \label{3.4}
\end{equation}

By the defination we have

\begin{equation*}
\left\langle T,\overset{-}{T}\right\rangle =0
\end{equation*}

By substituting the last equation in (\ref{3.4}) we get

\begin{equation}
1+\frac{d\lambda }{ds}=0~\text{and }\frac{d}{ds}\left( \mu \left( s\right)
+\varepsilon \mu ^{\ast }\left( s\right) \right) =-1  \label{3.5}
\end{equation}

The necessary operations are maken, we get

\begin{equation*}
\mu ^{^{\prime }}\left( s\right) =-1~\text{and }\mu ^{\ast ^{\prime }}\left(
s\right) =0
\end{equation*}

By taking the integral of the last equation we get

\begin{equation}
\mu \left( s\right) =c_{1}-s~\text{and }\mu ^{\ast }\left( s\right) =c_{2}
\label{3.6}
\end{equation}

By substituting (\ref{3.6}) in (\ref{3.2}) we get

\begin{equation}
\beta \left( s\right) -\alpha \left( s\right) =\left[ \left( c_{1}-s\right)
+\varepsilon c_{2}\right] T\left( s\right) .  \label{3.7}
\end{equation}
\end{proof}

\qquad \qquad

\begin{corollary}
The distance between the dual curves $\beta ~$and $\alpha ~$is $\left\vert
c_{1}-s\right\vert \mp \varepsilon c_{2}$.
\end{corollary}

\begin{proof}
By taking the norm of the equation (\ref{3.7}) we get\qquad

\begin{equation}
d\left( \alpha \left( s\right) ,\beta \left( s\right) \right) =\left\vert
c_{1}-s\right\vert \mp \varepsilon c_{2}  \label{3.8}
\end{equation}%
\qquad
\end{proof}

\begin{theorem}
Let $\alpha ,\beta ~$be dual curves. If the dual curve~$\beta $ involute of
the dual curve $\alpha $, The relationships between the dual frenet vectors
of the dual curves $\alpha ~$and~$\beta $
\end{theorem}

\begin{equation*}
\left\{ 
\begin{array}{c}
\overset{-}{T}=N \\ 
\overset{-}{N}=-\text{cos}\Phi T+\text{sin}\Phi B \\ 
\overset{-}{B}=\text{sin}\Phi T+\text{cos}\Phi B%
\end{array}%
\right.
\end{equation*}

\qquad \qquad

\begin{proof}
By differentiating the equation (\ref{3.2}) with respect to $s$ we obtain

\begin{equation}
\beta ^{^{\prime }}\left( s\right) =\lambda \kappa \left( s\right) N\left(
s\right) ~,~\lambda =\left( c_{1}-s\right) +\varepsilon c_{2}  \label{3.9}
\end{equation}

and

\begin{equation*}
\left\Vert \beta ^{^{\prime }}\left( s\right) \right\Vert =\lambda \kappa
\left( s\right)
\end{equation*}

Thus, the tangent vector of$~\beta $ is found

\begin{equation*}
\overset{-}{T}=\frac{\beta ^{^{\prime }}\left( s\right) }{\left\Vert \beta
^{^{\prime }}\left( s\right) \right\Vert }=\frac{\lambda \kappa \left(
s\right) N\left( s\right) }{\lambda \kappa \left( s\right) }
\end{equation*}

If we arrange the last equation we obtain

\begin{equation}
\overset{-}{T}=N\left( s\right)  \label{3.10}
\end{equation}

By differentiating the equation (\ref{3.9}) with respect to $s$ we obtain

\begin{equation*}
\beta ^{^{\prime \prime }}=-\lambda \kappa ^{2}T+\left( \lambda \kappa
^{^{\prime }}-\kappa \right) N+\lambda \kappa \tau B
\end{equation*}

If the cross product $\beta ^{^{\prime }}\wedge \beta ^{^{\prime \prime }}~$%
is calculated we have

\begin{equation}
\beta ^{^{\prime }}\wedge \beta ^{^{\prime \prime }}=\lambda ^{2}\kappa
^{2}\tau T+\lambda ^{2}\kappa ^{3}B  \label{3.11}
\end{equation}

The norm of vector $\beta ^{^{\prime }}\wedge \beta ^{^{\prime \prime }}~$is
found\qquad

\begin{equation}
\left\Vert \beta ^{^{\prime }}\wedge \beta ^{^{\prime \prime }}\right\Vert
=\lambda ^{2}\kappa ^{2}\sqrt{\kappa ^{2}+\tau ^{2}}  \label{3.12}
\end{equation}

For the dual binormal vector of the dual curve $\beta ~$we can write

\begin{equation*}
\overset{-}{B}=\frac{\beta ^{^{\prime }}\wedge \beta ^{^{\prime \prime }}}{%
\left\Vert \beta ^{^{\prime }}\wedge \beta ^{^{\prime \prime }}\right\Vert }
\end{equation*}

By substituting (\ref{3.11}) and (\ref{3.12}) in the last equation we get

\begin{equation}
\overset{-}{B}=\frac{\tau }{\sqrt{\kappa ^{2}+\tau ^{2}}}T+\frac{\kappa }{%
\sqrt{\kappa ^{2}+\tau ^{2}}}B  \label{3.13}
\end{equation}

For the dual principal normal vector of the dual curve $\beta $\ we can write

\begin{equation*}
\overset{-}{N}=\overset{-}{B}\wedge \overset{-}{T}
\end{equation*}

and

\begin{equation}
\overset{-}{N}=-\frac{\kappa }{\sqrt{\kappa ^{2}+\tau ^{2}}}T+\frac{\tau }{%
\sqrt{\kappa ^{2}+\tau ^{2}}}B  \label{3.14}
\end{equation}
\end{proof}

Let $\Phi ~\left( \Phi =\varphi +\varepsilon \varphi ^{\ast
}~\,,~\varepsilon ^{2}=0\right) $\ be dual angle between the dual darboux
vector $W$ of $\alpha ~$and dual unit binormal vector $B$ in this situation
we can write

\begin{equation}
\text{sin}\Phi =\frac{\tau }{\kappa ^{2}+\tau ^{2}}\,\,,~\ \text{cos}\Phi =%
\frac{\kappa }{\kappa ^{2}+\tau ^{2}}  \label{3.15}
\end{equation}

By substituting (\ref{3.15}) in (\ref{3.12}) and (\ref{3.13}) the proof is
completed.

If the equation

\begin{equation*}
\left\{ 
\begin{array}{c}
\overset{-}{T}=N \\ 
\overset{-}{N}=-\text{cos}\Phi T+\text{sin}\Phi B \\ 
\overset{-}{B}=\text{sin}\Phi T+\text{cos}\Phi B%
\end{array}%
\right.
\end{equation*}

is separated into the real and dual part, we can obtain

\begin{equation*}
\left\{ 
\begin{array}{c}
\overset{-}{t}=n \\ 
\overset{-}{n}=-\text{cos}\varphi t+\text{sin}\varphi b \\ 
\overset{-}{b}=\text{sin}\varphi t+\text{cos}\varphi b%
\end{array}%
\right.
\end{equation*}

\begin{equation*}
\left\{ 
\begin{array}{c}
\overset{-}{t}^{\ast }=n^{\ast } \\ 
\overset{-}{n}^{\ast }=-\text{cos}\varphi t^{\ast }+\text{sin}\varphi
b^{\ast }+\varphi ^{\ast }\left( \text{sin}\varphi t+\text{cos}\varphi
b\right) \\ 
\overset{-}{b}^{\ast }=\text{sin}\varphi t^{\ast }+\text{cos}\varphi b^{\ast
}+\varphi ^{\ast }\left( \text{cos}\varphi t-\text{sin}\varphi b\right)%
\end{array}%
\right.
\end{equation*}

On the way

\begin{equation*}
\left\{ 
\begin{array}{c}
\text{sin}\Phi =\text{sin}\left( \varphi +\varepsilon \varphi ^{\ast
}\right) =\text{sin}\varphi +\varepsilon \varphi ^{\ast }\text{cos}\varphi
\\ 
\text{cos}\Phi =\text{cos}\left( \varphi +\varepsilon \varphi ^{\ast
}\right) =\text{cos}\varphi -\varepsilon \varphi ^{\ast }\text{sin}\varphi%
\end{array}%
\right.
\end{equation*}%
\qquad \qquad

If the equation

\begin{equation*}
\text{sin}\Phi =\frac{\tau }{\kappa ^{2}+\tau ^{2}}
\end{equation*}

is separated into the real and dual part, we can obtain

\begin{equation*}
\left\{ 
\begin{array}{c}
\text{sin}\varphi =\frac{k_{2}}{k_{1}^{2}+k_{2}^{2}} \\ 
\text{cos}\varphi =\frac{k_{1}^{2}+k_{2}^{\ast }-2k_{1}k_{2}k_{1}^{\ast
}-2k_{2}^{2}k_{2}^{\ast }}{\varphi \left( k_{1}^{2}+k_{2}^{2}\right) ^{2}}%
\end{array}%
\right.
\end{equation*}%
\qquad \qquad \qquad

If the equation

\begin{equation*}
\text{cos}\Phi =\frac{\kappa }{\kappa ^{2}+\tau ^{2}}
\end{equation*}

is separated into the real and dual part, we can obtain\qquad \qquad

\begin{equation*}
\left\{ 
\begin{array}{c}
\text{cos}\varphi =\frac{k_{1}}{k_{1}^{2}+k_{2}^{2}} \\ 
\text{sin}\varphi =\frac{2k_{1}^{2}+k_{1}^{\ast }+2k_{1}k_{2}k_{2}^{\ast
}-k_{1}^{2}k_{1}^{\ast }-k_{2}^{2}k_{1}^{\ast }}{\varphi \left(
k_{1}^{2}+k_{2}^{2}\right) ^{2}}%
\end{array}%
\right.
\end{equation*}

\begin{theorem}
Let $\alpha ,\beta ~$be dual curves. If the dual curve~$\beta $ involute of
the dual curve $\alpha $ , curvature and torsion of the dual curve $\beta ~$%
are
\end{theorem}

\begin{equation}
\overset{-}{\kappa }^{2}\left( s\right) =\frac{\kappa ^{2}\left( s\right)
+\tau ^{2}\left( s\right) }{\lambda ^{2}\left( s\right) \kappa ^{2}\left(
s\right) }~,~\overset{-}{\tau }\left( s\right) =\frac{\kappa \left( s\right)
\tau ^{^{\prime }}\left( s\right) -\kappa ^{^{\prime }}\left( s\right) \tau
\left( s\right) }{\lambda \left( s\right) \kappa \left( s\right) \left(
\kappa ^{2}\left( s\right) +\tau ^{2}\left( s\right) \right) }  \label{3.16}
\end{equation}

\begin{proof}
By the defination of involute we can write

\begin{equation}
\beta \left( s\right) =\alpha \left( s\right) +\left\vert \lambda
\right\vert T\left( s\right)  \label{3.17}
\end{equation}

By differentiating the equation (\ref{3.17}) with respect to $s$ we obtain

\begin{eqnarray*}
\frac{d\beta }{ds^{\ast }}\frac{ds^{\ast }}{ds} &=&T\left( s\right)
+\left\vert \lambda \right\vert ^{^{\prime }}T\left( s\right) +\left\vert
\lambda \right\vert \kappa \left( s\right) N\left( s\right) \\
\frac{d\beta }{ds^{\ast }}\frac{ds^{\ast }}{ds} &=&T\left( s\right) -T\left(
s\right) +\left\vert \lambda \right\vert \kappa \left( s\right) N\left(
s\right)
\end{eqnarray*}

\begin{equation}
\overset{-}{T}\left( s\right) \frac{ds^{\ast }}{ds}=\left\vert \lambda
\right\vert \kappa \left( s\right) N\left( s\right)  \label{3.18}
\end{equation}

Since the direction of $\overset{-}{T}\left( s\right) ~$is coincident with $%
N\left( s\right) ~$we have

\begin{equation}
\overset{-}{T}\left( s\right) =N\left( s\right)  \label{3.19}
\end{equation}

Taking the inner product of (\ref{3.18}) with $T$ and necessary operation
are maken we get

\begin{equation}
\frac{ds^{\ast }}{ds}=\left\vert \lambda \left( s\right) \right\vert \kappa
\left( s\right)  \label{3.20}
\end{equation}

By taking derivative of (\ref{3.19}) and applying the frenet formulae (\ref%
{2.1}) we have

\begin{equation}
\overset{-}{T}\left( s\right) =N\left( s\right) \Rightarrow \overset{-}{T}%
^{^{\prime }}\left( s\right) \frac{ds^{\ast }}{ds}=-\kappa T+\tau B
\label{3.21}
\end{equation}

From (\ref{3.20}) and (\ref{3.21}) we have

\begin{equation*}
\overset{-}{T}^{^{\prime }}\left( s\right) =\frac{-\kappa T+\tau B}{%
\left\vert \lambda \left( s\right) \right\vert \kappa \left( s\right) }
\end{equation*}

From the last equation we can write

\begin{equation*}
\overset{-}{\kappa }\left( s\right) \overset{-}{N}\left( s\right) =\frac{%
-\kappa T+\tau B}{\left\vert \lambda \left( s\right) \right\vert \kappa
\left( s\right) }
\end{equation*}

Taking the inner product the last equation with each other we have

\begin{equation*}
\left\langle \overset{-}{\kappa }\left( s\right) \overset{-}{N}\left(
s\right) ,\overset{-}{\kappa }\left( s\right) \overset{-}{N}\left( s\right)
\right\rangle =\left\langle \frac{-\kappa T+\tau B}{\left\vert \lambda
\left( s\right) \right\vert \kappa \left( s\right) },\frac{-\kappa T+\tau B}{%
\left\vert \lambda \left( s\right) \right\vert \kappa \left( s\right) }%
\right\rangle
\end{equation*}

Thus, we find

\begin{equation*}
\overset{-}{\kappa }^{2}\left( s\right) =\frac{\kappa ^{2}\left( s\right)
+\tau ^{2}\left( s\right) }{\lambda ^{2}\left( s\right) \kappa ^{2}\left(
s\right) }
\end{equation*}

We know that

\begin{equation*}
\beta ^{^{\prime }}\wedge \beta ^{^{\prime \prime }}=\lambda ^{2}\kappa
^{2}\tau T+\lambda ^{2}\kappa ^{3}B
\end{equation*}

Taking the norm the last equation we get

\begin{equation*}
\left\Vert \beta ^{^{\prime }}\wedge \beta ^{^{\prime \prime }}\right\Vert
=\kappa ^{4}\lambda ^{4}\left( \kappa ^{2}+\tau ^{2}\right)
\end{equation*}

By substituting these equations in (\ref{2.3}) we get

\begin{equation*}
\overset{-}{\tau }=\frac{\left\vert 
\begin{array}{ccc}
0 & \kappa \lambda  & 0 \\ 
-\kappa ^{2}\lambda  & \left( \kappa \lambda \right) ^{^{\prime }} & \kappa
\tau \lambda  \\ 
\left( -\kappa ^{2}\lambda \right) ^{^{\prime }}-\kappa \left( \kappa
\lambda \right) ^{^{\prime }} & -\kappa ^{3}\lambda +\left( \kappa \lambda
\right) ^{^{\prime \prime }}-\kappa \tau ^{2}\lambda  & \left( \kappa
\lambda \right) ^{^{\prime }}\tau +\left( \kappa \tau \lambda \right)
^{^{\prime }}%
\end{array}%
\right\vert }{\left\Vert \beta ^{^{\prime }}\wedge \beta ^{^{\prime \prime
}}\right\Vert ^{2}}
\end{equation*}

\begin{equation*}
\overset{-}{\tau }=\frac{\kappa \tau ^{^{\prime }}-\kappa ^{^{\prime }}\tau 
}{\kappa \left\vert \lambda \right\vert \left( \kappa ^{2}+\tau ^{2}\right) }
\end{equation*}
\end{proof}

If the equation (\ref{3.16}) is separated into the real and dual part, we
can obtain

\begin{equation*}
\left\{ 
\begin{array}{c}
\overset{-}{k_{1}}=\frac{\sqrt{k_{1}^{2}+k_{2}^{2}}}{\mu k_{1}} \\ 
\overset{-}{k_{1}}^{\ast }=\frac{\left( \mu ^{2}k_{1}^{2}\right) \left(
2k_{1}k_{1}^{\ast }+2k_{2}k_{2}^{\ast }\right) -\left( 2k_{1}k_{1}^{\ast
}\mu ^{2}\right) \left( k_{1}^{2}+k_{2}^{2}\right) }{2\mu ^{3}k_{1}^{3}\sqrt{%
k_{1}^{2}+k_{2}^{2}}}%
\end{array}%
\right.
\end{equation*}

\begin{equation*}
\left\{ 
\begin{array}{c}
\overset{-}{k_{2}}=\frac{k_{1}k_{2}^{^{\prime }}-k_{2}k_{1}^{^{\prime }}}{%
\mu k_{1}\left( k_{1}^{2}+k_{2}^{2}\right) } \\ 
\overset{-}{k_{2}}^{\ast }=\frac{\left( k_{1}k_{2}^{^{\prime \ast
}}+k_{2}^{^{\prime }}k_{1}^{\ast }-k_{1}^{^{\prime }}k_{2}^{\ast
}-k_{2}k_{1}^{^{\prime \ast }}\right) \left( \mu k_{1}^{3}+k_{1}k_{2}^{2}\mu
\right) -\left[ 2\left( k_{1}k_{1}^{\ast }+k_{2}k_{2}^{\ast }\right)
k_{1}\mu +\left( k_{1}^{2}+k_{2}^{2}\right) \left( k_{1}^{\ast }\mu
+k_{1}\mu ^{\ast }\right) \right] \left( k_{1}k_{2}^{^{\prime
}}-k_{2}k_{1}^{^{\prime }}\right) }{\left( \mu k_{1}^{3}+k_{1}k_{2}^{2}\mu
\right) ^{2}}%
\end{array}%
\right.
\end{equation*}

\begin{theorem}
Let $\alpha ,\beta ~$be dual curves and the dual curve~$\beta $ involute of
the dual curve $\alpha $ . If $W$ and $\overset{-}{W}~$are darboux vectors
of the dual curves $\alpha ~$and $\beta ~$we can write
\end{theorem}

\begin{equation}
\overset{-}{W}=\frac{1}{\lambda \kappa }\left( W+\Phi ^{^{\prime }}N\right)
\label{3.22}
\end{equation}

\begin{proof}
Since ~$\overset{-}{W}~$is darboux vector of $\beta \left( s\right) $\ we
can write

\begin{equation}
\overset{-}{W}\left( s\right) =\overset{-}{\tau }\left( s\right) \overset{-}{%
T}\left( s\right) +\overset{-}{\kappa }\left( s\right) \overset{-}{B}\left(
s\right)  \label{3.23}
\end{equation}

By substituting $\overset{-}{\tau },\overset{-}{T},\overset{-}{\kappa },%
\overset{-}{B}~$in the last equation we get\qquad

\begin{equation}
\overset{-}{W}\left( s\right) =\frac{\kappa \tau ^{^{\prime }}-\kappa
^{^{\prime }}\tau }{\kappa \left\vert \lambda \right\vert \left( \kappa
^{2}+\tau ^{2}\right) }N\left( s\right) +\frac{\sqrt{\kappa ^{2}+\tau ^{2}}}{%
\kappa \left\vert \lambda \right\vert }\left( \text{sin}\Phi T+\text{cos}%
\Phi B\right)  \label{3.24}
\end{equation}

By substituting (\ref{3.15}) in (\ref{3.24}) we get\qquad

\begin{equation*}
\overset{-}{W}\left( s\right) =\frac{\kappa \tau ^{^{\prime }}-\kappa
^{^{\prime }}\tau }{\kappa \left\vert \lambda \right\vert \left( \kappa
^{2}+\tau ^{2}\right) }N\left( s\right) +\frac{\sqrt{\kappa ^{2}+\tau ^{2}}}{%
\kappa \left\vert \lambda \right\vert }\left( \frac{\tau T+\kappa B}{\sqrt{%
\kappa ^{2}+\tau ^{2}}}\right)
\end{equation*}

The necessary operation are maken, we get

\begin{equation*}
\overset{-}{W}\left( s\right) =\frac{\tau T+\kappa B}{\kappa \left\vert
\lambda \right\vert }+\frac{\kappa \tau ^{^{\prime }}-\kappa ^{^{\prime
}}\tau }{\kappa \left\vert \lambda \right\vert \left( \kappa ^{2}+\tau
^{2}\right) }N\left( s\right)
\end{equation*}

\begin{equation*}
\overset{-}{W}\left( s\right) =\frac{1}{\kappa \left\vert \lambda
\right\vert }\left( \tau T+\kappa B+\frac{\kappa \tau ^{^{\prime }}-\kappa
^{^{\prime }}\tau }{\kappa ^{2}+\tau ^{2}}N\right)
\end{equation*}

and

\begin{equation*}
\overset{-}{W}\left( s\right) =\frac{1}{\kappa \left\vert \lambda
\right\vert }\left( W+\frac{\left( \frac{\tau }{\kappa }\right) ^{^{\prime
}}\kappa ^{2}}{\kappa ^{2}+\tau ^{2}}N\right)
\end{equation*}

Furthermore, Since

\begin{equation*}
\frac{\text{sin}\Phi }{\text{cos}\Phi }=\frac{\tau \diagup \sqrt{\kappa
^{2}+\tau ^{2}}}{\kappa \diagup \sqrt{\kappa ^{2}+\tau ^{2}}}
\end{equation*}

\begin{equation*}
\frac{\tau }{\kappa }=\text{tan}\Phi 
\end{equation*}

By taking derivative of the last equation we have 

\begin{equation*}
\Phi ^{^{\prime }}\text{sec}^{2}\Phi =\left( \frac{\tau }{\kappa }\right)
^{^{\prime }}
\end{equation*}

The necessary operations are maken, we get

\begin{equation*}
\Phi ^{^{\prime }}=\left( \frac{\tau }{\kappa }\right) ^{^{\prime }}\frac{%
\kappa }{\kappa ^{2}+\tau ^{2}}
\end{equation*}

In this situation, the proof is completed

\begin{equation*}
\overset{-}{W}\left( s\right) =\frac{1}{\kappa \left\vert \lambda
\right\vert }\left( W+\Phi ^{^{\prime }}N\right) 
\end{equation*}
\end{proof}

If the equation (\ref{3.22}) is separated into the real and dual part, we
can obtain

\begin{equation*}
\left\{ 
\begin{array}{c}
\overset{-}{w}=\frac{w+\varphi ^{^{\prime }}n}{\mu k_{1}} \\ 
\overset{-}{w}^{\ast }=\frac{\mu k_{1}\left( w^{\ast }+\varphi ^{^{\prime
}}n+\varphi ^{^{\prime }\ast }n\right) -\left( \mu k_{1}^{\ast }+\mu ^{\ast
}k_{1}\right) \left( w+\varphi ^{^{\prime }}n\right) }{\mu ^{2}k_{1}^{2}}%
\end{array}%
\right. 
\end{equation*}%
\qquad \qquad If the equation (\ref{3.24}) is separated into the real and
dual part, we can obtain%
\begin{equation*}
\left\{ 
\begin{array}{c}
\overset{-}{w}=\frac{\sqrt{k_{1}^{2}+k_{2}^{2}}}{\mu k_{1}}\left( \text{sin}%
\varphi t+\text{cos}\varphi b\right)  \\ 
\overset{-}{w}^{\ast }=\frac{\sqrt{k_{1}^{2}+k_{2}^{2}}}{\mu k_{1}}\left( 
\text{sin}\varphi t^{\ast }+\text{cos}\varphi b^{\ast }+\varphi ^{\ast
}\left( \text{cos}\varphi t-\text{sin}\varphi b\right) \right) +\frac{\mu
k_{1}\left( k_{1}k_{1}^{\ast }+k_{2}k_{2}^{\ast }\right) -\left(
k_{1}^{2}+k_{2}^{2}\right) \left( \mu k_{1}^{\ast }+\mu ^{\ast }k_{1}\right) 
}{\sqrt{k_{1}^{2}+k_{2}^{2}}\mu ^{2}k_{1}^{2}}\left( \text{sin}\varphi t+%
\text{cos}\varphi b\right) 
\end{array}%
\right. 
\end{equation*}
\ 

\begin{theorem}
\bigskip Let $\alpha ,\beta ~$be dual curves and the dual curve~$\beta $
involute of the dual curve $\alpha $ . If $C$ and $\overset{-}{C}~$are unit
vectors of the direction of  $~W~$and$\overset{-}{~W}$ , respectively
\end{theorem}

\begin{equation}
\overset{-}{C}=\frac{\Phi ^{^{\prime }}}{\sqrt{\Phi ^{^{\prime }2}+\kappa
^{2}+\tau ^{2}}}N+\frac{\sqrt{\kappa ^{2}+\tau ^{2}}}{\sqrt{\Phi ^{^{\prime
}2}+\kappa ^{2}+\tau ^{2}}}C  \label{3.25}
\end{equation}

\begin{proof}
Since $\beta ~$the dual angle between $\overset{-}{~W}~$and $\overset{-}{B}~$%
we can write

\begin{equation*}
\overset{-}{C}\left( s\right) =\text{sin}\beta \overset{-}{T}\left( s\right)
+\text{cos}\beta \overset{-}{B}\left( s\right) 
\end{equation*}

In here, we want to find the statements sin$\beta ~$and cos$\beta ,$

We know that

\begin{equation*}
\text{sin}\beta =\frac{\overset{-}{\tau }}{\left\Vert \overset{-}{W}%
\right\Vert }=\frac{\overset{-}{\tau }}{\sqrt{\overset{-}{\kappa }^{2}+%
\overset{-}{\tau }^{2}}}
\end{equation*}

By substituting  $\overset{-}{\tau }~$and $\overset{-}{\kappa }~$in the last
equation and necessary operatios are maken, we get

\begin{equation}
\text{sin}\beta =\frac{\Phi ^{^{\prime }}}{\sqrt{\Phi ^{^{\prime }2}+\kappa
^{2}+\tau ^{2}}}  \label{3.26}
\end{equation}

Similary,

\begin{equation}
\text{cos}\beta =\frac{\sqrt{\kappa ^{2}+\tau ^{2}}}{\sqrt{\Phi ^{^{\prime
}2}+\kappa ^{2}+\tau ^{2}}}  \label{3.27}
\end{equation}

Thus we find

\begin{equation*}
\overset{-}{C}=\frac{\Phi ^{^{\prime }}}{\sqrt{\Phi ^{^{\prime }2}+\kappa
^{2}+\tau ^{2}}}\overset{-}{T}+\frac{\sqrt{\kappa ^{2}+\tau ^{2}}}{\sqrt{%
\Phi ^{^{\prime }2}+\kappa ^{2}+\tau ^{2}}}C
\end{equation*}
\end{proof}

If the equation (\ref{3.25}) is separated into the real and dual part, we
can obtain

\begin{equation*}
\left\{ 
\begin{array}{c}
\overset{-}{c}=\frac{\varphi ^{^{\prime }}n+\sqrt{k_{1}^{2}+k_{2}^{2}}c}{%
\sqrt{\varphi ^{^{\prime }}+k_{1}^{2}+k_{2}^{2}}} \\ 
\overset{-}{c}^{\ast }=\frac{\varphi ^{^{\prime }}n^{\ast }+\varphi
^{^{\prime ^{\ast }}}n+\sqrt{k_{1}^{2}+k_{2}^{2}}c^{\ast }+\frac{%
k_{1}k_{1}^{\ast }+k_{2}k_{2}^{\ast }}{\sqrt{k_{1}^{2}+k_{2}^{2}}}c-\frac{%
\varphi ^{^{\prime }}n\left( \sqrt{k_{1}^{2}+k_{2}^{2}}\right) c\left(
\varphi ^{^{\prime }}\varphi ^{^{\prime ^{\ast }}}+k_{1}k_{1}^{\ast
}+k_{2}k_{2}^{\ast }\right) }{\sqrt{\varphi ^{^{\prime }}+k_{1}^{2}+k_{2}^{2}%
}}}{\sqrt{\varphi ^{^{\prime }}+k_{1}^{2}+k_{2}^{2}}}%
\end{array}%
\right. 
\end{equation*}

If the equation (\ref{3.26}) and (\ref{3.27}) are separated into the real
and dual part, we can obtain

\begin{equation*}
\left\{ 
\begin{array}{c}
\text{sin}\overset{-}{\varphi }=\frac{\varphi ^{^{\prime }}}{\sqrt{\varphi
^{^{\prime }}+k_{1}^{2}+k_{2}^{2}}} \\ 
\text{cos}\overset{-}{\varphi }=\frac{\left( \Phi ^{^{\prime }2}+\kappa
^{2}+\tau ^{2}\right) \Phi ^{^{\prime }\ast }-\varphi ^{^{\prime }}\varphi
^{^{^{\ast }}}+k_{1}k_{1}^{\ast }+k_{2}k_{2}^{\ast }\varphi ^{^{\prime }}}{%
\overset{-}{\varphi }^{\ast }\left( \Phi ^{^{\prime }2}+\kappa ^{2}+\tau
^{2}\right) ^{\frac{3}{2}}}%
\end{array}%
\right. 
\end{equation*}

\begin{equation*}
\left\{ 
\begin{array}{c}
\text{cos}\overset{-}{\varphi }=\sqrt{\frac{k_{1}^{2}+k_{2}^{2}}{\varphi
^{^{\prime ^{2}}}+k_{1}^{2}+k_{2}^{2}}} \\ 
\text{sin}\overset{-}{\varphi }=\frac{\left( \varphi ^{^{\prime }}\varphi
^{^{\prime ^{\ast }}}+k_{1}k_{1}^{\ast }+k_{2}k_{2}^{\ast }\right) \sqrt{%
k_{1}^{2}+k_{2}^{2}}-\left( \varphi ^{^{\prime
^{2}}}+k_{1}^{2}+k_{2}^{2}\right) \left( k_{1}k_{1}^{\ast }+k_{2}k_{2}^{\ast
}\right) }{\overset{-}{\varphi }^{\ast }\left( \Phi ^{^{\prime }2}+\kappa
^{2}+\tau ^{2}\right) ^{\frac{3}{2}}\sqrt{k_{1}^{2}+k_{2}^{2}}}%
\end{array}%
\right. 
\end{equation*}%
\qquad \qquad 

\begin{corollary}
Let $\alpha ,\beta ~$be dual curves and the dual curve~$\beta $ involute of
the dual curve $\alpha $ . If evolute curve $\alpha ~$is helix,
\end{corollary}

i)\qquad The vectors $\overset{-}{~W}~$and $\overset{-}{~B}~$of the involute
curve $\beta $\ are linearly dependent.

ii)\qquad $C$=$\overset{-}{C}$

iii)\qquad $\beta $\ is planar.

\begin{proof}
i) If the evolute\ curve $\alpha $\ is helix, then we have

\begin{equation*}
\frac{\tau }{\kappa }=\text{tan}\Phi =cons\text{ or }\Phi ^{^{\prime }}=0
\end{equation*}

and then we have

\begin{equation}
\left\{ 
\begin{array}{c}
\text{sin}\overset{-}{\Phi }=0 \\ 
\text{cos}\overset{-}{\Phi }=1%
\end{array}%
\right.   \label{3.28}
\end{equation}

Thus, we get

\begin{equation}
\overset{-}{\Phi }=0  \label{3.29}
\end{equation}

ii) Substituting by the equation (\ref{3.29}) into the equation (\ref{3.25})
, we have

\begin{equation*}
C=\overset{-}{C}
\end{equation*}

iii ) For being is a helix , then we have

\begin{equation*}
\frac{\tau }{\kappa }=cons
\end{equation*}

\begin{equation}
\left( \frac{\tau }{\kappa }\right) ^{^{\prime }}=0  \label{3.30}
\end{equation}

On the other hand, from the equation (\ref{3.16}) , we can write

\begin{equation*}
\frac{\overset{-}{\tau }}{\overset{-}{\kappa }}=\frac{\frac{\kappa \tau
^{^{\prime }}-\kappa ^{^{\prime }}\tau }{\lambda \kappa \left( \kappa
^{2}+\tau ^{2}\right) }}{\frac{\left( ^{\kappa ^{2}+\tau ^{2}}\right) ^{%
\frac{1}{2}}}{\lambda \kappa }}
\end{equation*}

and

\begin{equation}
\frac{\overset{-}{\tau }}{\overset{-}{\kappa }}=\frac{\left( \frac{\tau }{%
\kappa }\right) ^{^{\prime }}\kappa ^{2}}{\left( ^{\kappa ^{2}+\tau
^{2}}\right) ^{\frac{3}{2}}}  \label{3.31}
\end{equation}

Substituting by the equation (\ref{3.30}) into the equation (\ref{3.31})
,then we find

\begin{equation*}
\overset{-}{\tau }=0
\end{equation*}
\end{proof}

\end{document}